\documentclass[10pt]{amsart}

\usepackage{amsmath, amsthm}
\usepackage{eucal}

\usepackage{amssymb}
\usepackage{amscd}
\usepackage{palatino}
\usepackage{latexsym}
\usepackage{epsfig}
\usepackage{graphicx}
\usepackage{amsfonts}
\usepackage{psfrag}
\usepackage{color}
\usepackage{curves}

\usepackage{url}
\usepackage{cite}

\usepackage[margin=1in]{geometry}

\usepackage{hyperref}

\input xy
\usepackage{caption}

\usepackage{array}
\newtheorem*{Th}{Theorem}
\newtheorem{theorem}{Theorem}[section]
\newtheorem{lemma}[theorem]{Lemma}

\newtheorem{corollary}[theorem]{Corollary}

\theoremstyle{definition}
\newtheorem{example}[theorem]{Example}
\newtheorem{definition}[theorem]{Definition}

\numberwithin{equation}{section}
\numberwithin{figure}{section}

\newcommand{\C}{{\mathbb{C}}}

\newcommand{\Z}{{\mathbb{Z}}}

\newcommand{\R}{{\mathbb{R}}}

\newcommand{\mf}{\mathbf}
\newcommand{\mfc}{\mf{c}}
\newcommand{\mfl}{\mf{\ell}}

\newcommand{\al}{\alpha}

\newcommand{\de}{\delta}
\newcommand{\lee}{\langle}
\newcommand{\ree}{\rangle}
\newcommand{\la}{\lambda}

\begin{document}

\title{Grossberg-Karshon twisted cubes and hesitant walk avoidance} 

\author{Megumi Harada}
\address{Department of Mathematics and
Statistics\\ McMaster University\\ 1280 Main Street West\\ Hamilton, Ontario L8S4K1\\ Canada}
\email{Megumi.Harada@math.mcmaster.ca}
\urladdr{\url{http://www.math.mcmaster.ca/Megumi.Harada/}}

\author{Eunjeong Lee}
\address{Department of Mathematical Sciences
\\ KAIST \\ 291 Daehak-ro Yuseong-gu \\ Daejeon 305-701 \\ South Korea}
\email{EunjeongLee@kaist.ac.kr}


\date{\today}

%%%%%%%%%%%%%%%%%%%%%
%  Abstract
%%%%%%%%%%%%%%%%%%%%%

%%%%%%%%%%%%%%%%%%%%%%%%%%%%%%%%%%%%%%%%%%%%%%%%%%%%%%%%%%%%%%%%%%%%%%%%
\begin{abstract}
  Let $G$ be a complex semisimple simply connected linear algebraic
  group. Let $\lambda$ be a dominant weight for $G$ and $\mathcal{I} =
  (i_1, i_2, \ldots, i_n)$ a word decomposition for an element $w =
  s_{i_1} s_{i_2} \cdots s_{i_n}$ of the Weyl group of $G$, where
  the $s_i$ are the simple reflections. 
  In the 1990s, Grossberg and
  Karshon introduced a virtual lattice polytope associated to
  $\lambda$ and $\mathcal{I}$, which they called a \textbf{twisted
    cube}, whose lattice points encode (counted with sign according to
  a density function) characters of representations of $G$. In recent
  work, the first author and Jihyeon Yang prove that the Grossberg-Karshon
  twisted cube is untwisted (so the support of the density function is
  a closed convex polytope) precisely when a certain torus-invariant
  divisor on a toric variety, constructed from the data of $\lambda$
  and $\mathcal{I}$, is basepoint-free. This corresponds to the situation in which the Grossberg-Karshon character formula is a 
true combinatorial formula, in the sense that there are no terms appearing with a minus sign. In this note, we translate
  this toric-geometric condition to the combinatorics of
  $\mathcal{I}$ and $\lambda$. More precisely, we introduce the
  notion of \textbf{hesitant $\lambda$-walks} and then prove that 
the associated Grossberg-Karshon twisted cube is untwisted precisely
when $\mathcal{I}$ is hesitant-$\lambda$-walk-avoiding. 
\end{abstract}

\maketitle

\setcounter{tocdepth}{1} \tableofcontents

%%%%%%%%%%%%%%%%%%%%%%%%%%%%%%%%%%%%%%%%%%%%%%%%%%%%%%%%%%%%%%%%%%%%%%%%

\section*{Introduction}

Let $G$ be a complex semisimple simply connected linear algebraic
group. 
Building combinatorial models for $G$-representations is a fruitful
technique in modern representation theory; a famous example is the
theory of crystal bases and string polytopes. In a different
direction, given a dominant weight $\lambda$ and a choice of word
expression $\mathcal{I} = (i_1, i_2, \ldots, i_n)$ of an
element $w = s_{i_1} s_{i_2} \cdots s_{i_n}$ in the Weyl group, Grossberg and Karshon \cite{Grossberg-Karshon} introduced
a combinatorial object called a \textbf{twisted cube} $(C(\mfc,
\mfl),\rho)$ where $C(\mfc, \mfl)$ is a subset of $\R^n$ and $\rho$ is
a support function with support precisely $C(\mfc,\mfl)$. The lattice
points of $C(\mfc, \mfl)$ encode (counted with $\pm$ sign according to $\rho$)
the character of the $G$-representation $V_\lambda$ \cite[Theorems 5
and 6]{Grossberg-Karshon}. Here the parameters $\mfc$ and $\mfl$ are
determined from $\lambda$ and $\mathcal{I}$. 
These
twisted cubes are combinatorially much simpler than general string
polytopes but they are not ``true'' polytopes in the sense that their
faces may have various angles and the intersection of faces may not be
a face (cf. \cite[\S 2.5 and Figure \ref{figure:first example} therein]{Grossberg-Karshon}),
and in general they may be neither closed nor convex (see
Example~\ref{example}). 
In particular, the Grossberg-Karshon character formula is not 
a purely combinatorial `positive' formula, since it may involve minus
signs.

The main result of this note gives necessary and sufficient conditions
on a dominant weight $\lambda$ and a (not necessarily reduced) word expression
$\mathcal{I} = (i_1, \ldots, i_n)$ of an element $w \in W$, 
such that the associated Grossberg-Karshon twisted
cube is untwisted (cf. Definition~\ref{definition:untwisted}), i.e.,
$C(\mfc, \mfl)$ is a closed, convex polytope and $\rho$ is identically
equal to $1$ on $C(\mfc,\mfl)$. This is precisely the situation in
which the Grossberg-Karshon character formula is a `true'
combinatorial formula, in the sense that it is a purely `positive'
formula (with no terms appearing with a minus sign). 

In order to state our result it is useful to introduce some
terminology (see Section~\ref{sec:walks} for details). Roughly, we say
that a word $\mathcal{I} = (i_1, \ldots, i_n)$ is a 
\textbf{diagram walk} (or simply \textbf{walk}) if successive 
roots are adjacent in the Dynkin diagram: for instance, in type $A_5$

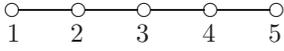
\begin{picture}(35,10)                   
                       \put(5,5){\circle{5}}    
                       \put(7.5,5){\line(1,0){20}}
                       \put(30,5){\circle{5}}
                      \put(32.5,5){\line(1,0){20}}
                       \put(55,5){\circle{5}}    
                       \put(57.5,5){\line(1,0){20}}
                       \put(80,5){\circle{5}}    
                       \put(82.5,5){\line(1,0){20}}
                       \put(105,5){\circle{5}}

                      \put(3,-7){$1$}
                      \put(27,-7){$2$}
                      \put(52,-7){$3$}
                      \put(77,-7){$4$}
                      \put(102,-7){$5$}
\end{picture} 

\medskip
\noindent the word $\mathcal{I} = (2,4,5)$ with corresponding simple roots $(s_2, s_4, s_5)$ 
is not a walk since $s_2$ and $s_4$ are
not adjacent, but $\mathcal{I} = (1,2,3,2,1)$ is a walk. Moreover, given a
dominant weight $\lambda = \lambda_1 \varpi_1 + \cdots + \lambda_r
\varpi_r$ written as a linear combination of the fundamental weights
$\{\varpi_1, \ldots, \varpi_r\}$, we say $\mathcal{I} = (i_1, i_2,
\ldots, i_n)$ is a
$\lambda$-walk if it is a walk and if it ends at a root which appears
in $\lambda$, i.e. $\lambda_{i_n} > 0$. A
\textbf{hesitant $\lambda$-walk} is a word $\mathcal{I} = (i_0, 
i_1, \ldots, i_n)$ where $i_0 = i_1$, so there is a repetition at
the first step, and the subword $(i_1, i_2, \ldots, i_n)$ is a
$\lambda$-walk. Finally, a word is
\textbf{hesitant-$\lambda$-walk-avoiding} if there is no subword which
is a hesitant $\lambda$-walk. With this terminology we can state the
main result of this paper.

\begin{Th}
Let $\mathcal{I} = (i_1, i_2, \ldots, i_n)$ be a word
decomposition of an element $w = s_{i_1} s_{i_2} \cdots s_{i_n}$ 
of the Weyl group $W$ 
and let 
$\lambda = \lambda_1 \varpi_1 + \lambda_2 \varpi_2 + \cdots + \lambda_r \varpi_r$ be a
dominant weight. 
Then the corresponding Grossberg-Karshon twisted cube $(C(\mfc, \mfl),
\rho)$ is untwisted if
and only if $\mathcal{I}$ is hesitant-$\lambda$-walk-avoiding. 
\end{Th}

We note that ``pattern avoidance'' is an important notion in the study of Schubert
varieties and Schubert calculus, first pioneered by Lakshmibai and Sandhya 
\cite{LakshmibaiSandhya} and further studied by many others (see e.g. 
\cite{AbeBilley} and references therein). It would be interesting to
explore the relation between our notion of
hesitant-$\lambda$-walk-avoidance with the other types of pattern
avoidance in the theory of flag and Schubert varieties.

We additionally remark that Kiritchenko recently has defined \emph{divided
  difference operators $D_i$ on polytopes} and, using these $D_i$ 
inductively together with a fixed choice of reduced word decomposition
for the longest element in the Weyl group of $G$, she constructs (possibly virtual) polytopes whose lattice
points encode the character of irreducible $G$-representations
\cite[Theorem 3.6]{Kiritchenko}.
Kiritchenko's virtual polytopes are generalizations of 
both Gel'fand-Cetlin polytopes and the 
Grossberg-Karshon twisted polytopes. 
It would be interesting to explore whether our methods can be
further generalized to study Kiritchenko's virtual polytopes (see
Section~\ref{sec:questions}).

This paper is organized as follows. In Section~\ref{sec:background} we
recall the necessary definitions and background from previous
papers. In particular, we recall the results of the first author and
Jihyeon Yang \cite[Proposition 2.1 and Theorem 2.4]{Harada-Yang:2014a} which characterize the
untwistedness of the Grossberg-Karshon twisted cube in terms of the
Cartier data associated to a certain toric divisor on a toric
variety; this is a key tool for our proof. 
In Section~\ref{sec:walks} we introduce the notions of
diagram walks and hesitant $\lambda$-walks and make the statement of
our main theorem. 
We prove the 
the necessity of hesitant-$\lambda$-walk-avoidance in
Section~\ref{sec:necessity}. The proof of sufficiency, which occupies
Section~\ref{sec:sufficiency}, is in part a case-by-case analysis according to
Lie type. We briefly record some open questions 
in Section~\ref{sec:questions}.

\medskip

\noindent \textbf{Acknowledgements.} 
We thank Jihyeon
Jessie Yang for useful conversations and Professor Dong Youp Suh for his support throughout the project. 
The first author was partially supported by an NSERC Discovery Grant (Individual),
an Ontario Ministry of Research
and Innovation Early Researcher Award, a Canada Research Chair
(Tier 2) award, and a Japan Society for the Promotion of Science
Invitation Fellowship for Research in Japan (Fellowship ID
L-13517). 
The second author was supported by the National Research Foundation of Korea (NRF) Grant funded by the Korean Government (MOE) (No. NRF-2013R1A1A2007780).
The first author additionally thanks the Osaka City
University Advanced Mathematics Institute for its hospitality while
part of this research was conducted.

\section{Background}\label{sec:background}

We begin by recalling the definition of \textbf{twisted cubes} given
by Grossberg and Karshon \cite[\S 2.5]{Grossberg-Karshon}. We follow the exposition in \cite{Harada-Yang:2014a}.
Fix a
positive integer $n$. A twisted
cube is a pair $(C(\mfc, \mfl), \rho)$ where $C(\mfc, \mfl)$
is a subset of $\R^n$ and $\rho: \R^n \to \R$ is a density function
with support precisely equal to $C(\mfc, \mfl)$. Here $\mfc =
\{c_{jk}\}_{1 \leq j < k \leq n}$ and
$\mfl = \{\ell_1, \ell_2, \ldots, \ell_n\}$ are fixed integers. (The general definition in \cite{Grossberg-Karshon} only requires the $\ell_i$ to be real numbers, but since we restrict attention to the cases arising from representation theory, our $\ell_i$ will always be integers.) 
In order to simplify the notation in what follows, we define
the following functions on $\R^n$: 
\begin{equation}\label{eq:def-A} 
\begin{split} 
A_n(x) = A_n(x_1, \ldots, x_n) & = \ell_n \\
A_j(x) = A_j(x_1, \ldots, x_n) & = \ell_j - \sum_{k > j} c_{jk} x_k
\textup{ for all } 1 \leq j \leq n-1. 
\end{split} 
\end{equation}
We also define a function $\textup{sgn}: \R \to \{\pm 1\}$ by
$\textup{sgn}(x) = 1$ for $x<0$ and $\textup{sgn}(x) = -1$ for $x
\geq 0$. 

We now give the precise definition. 

\begin{definition}\label{definition:twisted cube}
Let $n, \mfc, \mfl$ and $A_j$ be as above. Let $C(\mfc, \mfl)$ denote
the following subset of $\R^n$: 
\begin{equation}
  \label{eq:defC}
  C(\mfc, \mfl) := \{ x = (x_1, \ldots, x_n) \in \R^n \mid \textup{
    for all } 1 \leq j \leq n, A_j(x) < x_j < 0 \textup{ or } 0 \leq
  x_j \leq A_j(x) \} \subseteq \R^n. 
\end{equation}
Moreover we define a density function $\rho: \R^n \to \R$ by 
\begin{equation}
  \label{eq:def-rho}
  \rho(x) =
  \begin{cases}
    (-1)^n \prod_{k=1}^n \textup{sgn}(x_k) & \textup{ if } x \in
    C(\mfc, \mfl) \\
   0 & \textup{ else.} 
  \end{cases}
\end{equation}
Evidently $\textup{supp}(\rho) = C(\mfc, \mfl)$. We call the pair
$(C(\mf, \mfl), \rho)$ the \textbf{twisted cube associated to $\mfc$
  and $\mfl$.} 
\end{definition}

A twisted cube may not be a cube in
the standard sense. In particular, the set $C$ may be neither convex
nor closed, as the following example shows. See also the discussion in \cite[\S 2.5]{Grossberg-Karshon}.

\begin{example}\label{example}
  Let $n=2$ and let $\mfl = (\ell_1 = 3, \ell_2 = 5)$ and $\mfc =
  \{c_{12} = 1\}$. Then
\[
C = \{ (x_1, x_2) \in \R^2 \, \mid \, 0 \leq x_2 \leq 5 \textup{ and }
( 3 - x_2 < x_1 < 0 \textup{ or } 0 \leq x_1 \leq 3 - x_2 ) \}.
\]
See Figure~\ref{figure:first example}. The value of the density function $\rho$ is recorded within each region.

\setlength{\unitlength}{0.5cm}
\begin{figure}[h]
\centering
\begin{minipage}{1\textwidth}
\centering
%\parbox{2cm}{$N = \Z^2$} &
\begin{picture}(4,6)(0,0)
\put(-2,0){\vector(1,0){7}}
\put(5.1, -0.1){\tiny{$x_1$}}
\curve(0,-1, 0,0)
\put(0,5){\vector(0,1){1}}
\put(-0.2,6.2){\tiny{$x_2$}}
\put(3,-0.2){$\bullet$}\put(3,-0.7){\tiny{$3$}}
\put(-0.2,5){$\circ$}\put(0.2,5.1){\tiny{$5$}}
\put(-0.2,3){$\bullet$}
\put(-2.2,5){$\circ$}
\thicklines
\curve(-1.9,5.2,-0.2,5.2)
\curve(0,3.1,3.1,0)
\curve(0,0,3.1,0)
\curve(0,3.1,0,0)
\multiput(0,3.1)(0,0.4){5}{\line(0,1){0.2}}
\curve(-0.2,3.3,-0.4,3.5)\curve(-0.6,3.7,-0.8,3.9)
\curve(-1,4.1,-1.2,4.3)\curve(-1.4,4.5,-1.6,4.7)
\curve(-1.8,4.9,-1.9,5)
\put(0.7,1){\tiny{$+1$}}
\put(-0.8,4.3){\tiny{$-1$}}

\end{picture}

\end{minipage}

\captionof{figure}{}\label{figure:first example}

\end{figure}

Note in particular that $C$ does \emph{not} contain the points $\{ (0, x_2) \mid 3 <
x_2 < 5 \}$ and the points $\{ (x_1, x_2) \, \mid \, 3 < x_2 < 5
\textup{ and } x_1 = 3 - x_2 \}$, so $C$ is not closed, and it is
also not convex. 
\end{example}

As mentioned in the introduction, the main goal of this note is to give necessary and sufficient conditions for the \emph{untwistedness} of the twisted cube, stated in terms of the combinatorics of the defining parameters. The following makes the notion precise. 

\begin{definition}\label{definition:untwisted} (cf. \cite[Definition 2.2]{Harada-Yang:2014a})
We say that 
Grossberg-Karshon twisted cube $(C=C(\mfc, \mfl), \rho)$ is
\textbf{untwisted} if 
$C$ is a closed convex
polytope, and the support for $\rho$ is
constant and equal to $1$ on $C$ and $0$ elsewhere. We say the twisted
cube is \textbf{twisted} if it is not untwisted. 
\end{definition}

The main result of \cite{Harada-Yang:2014a} characterizes the untwistedness of the Grossberg-Karshon twisted cube in terms of the basepoint-freeness of a certain toric divisor on a toric variety constructed from the data of $\mfc$ and $\mfl$, which in turn can be stated in terms of the so-called Cartier data $\{m_\sigma\}$ associated to the divisor. In particular, in this paper we will not require the geometric perspective; instead we work with the integer vectors $m_\sigma$, which can be derived directly from the constants $\mfc$ and $\mfl$. Before quoting the relevant result from \cite{Harada-Yang:2014a} we need some terminology.

Let $\{e_1^+,\dots,e_n^+\}$ be the standard basis of $\R^n$. 
For $\sigma = (\sigma_1,\dots,\sigma_n) \in% 
\{+,-\}^n$, define $m_{\sigma} = (m_{\sigma,1}, \ldots, m_{\sigma,n}) = \sum m_{\sigma,k} e_k^+ \in \Z^n$ as 
follows, using the functions $A_k(x)$ defined in~\eqref{eq:def-A}. 
\begin{equation}\label{eq:def m_sigma}
m_{\sigma,k} = 
\begin{cases}
0 & \textrm{ if } \sigma_k = + \\
A_k(m_{\sigma,k+1},\dots,m_{\sigma,n}) & \textrm{ if } \sigma_k = -.
\end{cases}
\end{equation}
We will also need a certain polytope $P_D$ as follows:  
\begin{equation}
  \label{eq:def P_D}
 P_D = \{ x \in \R^n \, \mid \, 0 \leq x_j \leq
  A_j(x) \textup{ for all } 1 \leq j \leq n \} \subseteq \R^n. 
\end{equation}

With this notation in place we can quote the following.

\begin{theorem}\label{thm-HY} (cf. \cite[Proposition 2.1]{Harada-Yang:2014a})
  Let $n, \mfc$ and $\mfl$ be as above and let $(C(\mfc, \mfl),\rho)$ denote the
  corresponding Grossberg-Karshon twisted polytope. 
Then $(C(\mfc, \mfl), \rho)$ is untwisted if and only if 
$m_{\sigma, k} \geq 0$ for all $\sigma \in \{+,-\}^n$ and for 
all $k$ with $1 \leq k \leq n$.  
\end{theorem}

Recall that the goal of this note is to analyze the case when the
defining parameters for the Grossberg-Karshon twisted polytope arise
from certain representation-theoretic data. We now briefly describe
how to derive the $\mfc$ and $\mfl$ in this case. 

Following \cite{Grossberg-Karshon}, let $G$ be a complex semisimple
simply-connected linear algebraic group
of rank $r$ over an
algebraically closed field $\mf{k}$.
Choose a Cartan subgroup $H\subset G$, and a Borel subgroup
Let $\{\al_1,\ldots, \al_r\}$ denote the simple roots,
$\{\al_1^{\vee}, \ldots, \al_r^{\vee}\}$ the coroots, 
and $\{\varpi_1, \ldots, \varpi_r\}$ the fundamental weights (characterized by the relation
$\langle\varpi_i,\al_j^{\vee}\rangle=\de_{ij}$). Let $s_{\alpha_i} \in
W$ denote the simple reflection in the Weyl group corresponding to the
root $\alpha_i$.

Fix a choice $\lambda = \lambda_1 \varpi_1 + \dots + \lambda_r
\varpi_r$ in the weight lattice, where $\lambda_i \in \Z$. 
Let $\mathcal{I} = (i_1, \ldots, i_n)$ be a sequence of elements in
$[r]:=\{1,2,\ldots,r\}$; this corresponds to a (not necessarily
reduced) decomposition of an element $w = s_{\al_{i_1}} s_{\al_{i_2}}
\cdots s_{\al_{i_n}}$ in $W$. For simplicity, we introduce the notation $\beta_j := \alpha_{i_j}$, so $\beta_j$ is the $j$-th simple root 
appearing in the word decomposition. 
For such $\lambda$ and $\mathcal{I}$
we define constants $\mfc, \mfl$ by the formulas (cf. \cite[\S 3.7]{Grossberg-Karshon})
\begin{equation}
  \label{eq:def cjk rep}
  c_{jk}=\lee \beta_k, \beta_j^{\vee} \ree %\al_{i_k}, \al_{i_j}^{\vee}\ree
\end{equation}
for $1 \leq j < k \leq n$, and
\begin{equation}
  \label{eq:def ell rep}
\ell_1=\lee \lambda,\beta_1^{\vee}\ree, \dots, \ell_n=\lee \lambda,\beta_n^{\vee}\ree.
\end{equation}
Note that if the
$j$-th simple reflection in the given
 word decomposition
$\mathcal{I}$ is equal to $\alpha_i$, then $\ell_j = \la_i$, and that
the constants $c_{jk}$ are matrix entries in the Cartan matrix of
$G$.

The following example illustrates these definitions.

\begin{example}
  Consider $G = SL(3,\C)$ with positive roots $\{\alpha_1,
  \alpha_2\}$. Let $\lambda = 2 \varpi_1 + \varpi_2$ and $\mathcal{I} = (1,2,1)$. 
Then 
  $(\beta_1, \beta_2, \beta_3) = (\alpha_1, \alpha_2, \alpha_1)$ and we
  have
  \begin{equation}
   \begin{split}
    c_{12} & = \langle \alpha_2, \alpha_1^{\vee} \rangle = -1 \\
    c_{13} & = \langle \alpha_1, \alpha_1^{\vee} \rangle = 2 \\
    c_{23} & = \langle \alpha_1, \alpha_2^{\vee} \rangle = -1 \\
    \mfl = (\ell_1, \ell_2, \ell_3) & = (\langle \lambda, \alpha_1^{\vee} \rangle = 2, \langle \lambda, \alpha_2^{\vee} \rangle = 1, \langle \lambda, \alpha_1^{\vee} \rangle = 2).
    \\
   \end{split}
  \end{equation}
\end{example}

As mentioned in the introduction, in the setting above Grossberg and Karshon derive a Demazure-type character formula for the irreducible $G$-representation corresponding to $\lambda$, expressed as a sum over the lattice points $\Z^n \cap C(\mfc, \mfl)$ in the Grossberg-Karshon twisted cube $(C(\mfc, \mfl),\rho)$ \cite[Theorem 5 and Theorem 6]{Grossberg-Karshon}. The lattice points appear with a plus or minus sign according the density function $\rho$. Hence their formula 
is a \emph{positive} formula if $\rho$ is constant and equal to $1$
 on all
of $C(\mfc,\mfl)$. 
From the point of view of representation theory it
is therefore of interest to determine conditions on the weight
$\lambda$ and the
 word decomposition $\mathcal{I} = (i_1,i_2,\dots,i_n)$ for an element $w = s_{i_1}s_{i_2}\dots s_{i_n}$ 
such that the associated
Grossberg-Karshon twisted cube is in fact untwisted. 
This is the motivation for this note.

\section{Diagram walks, hesitant walk avoidance, and the statement of the main theorem}\label{sec:walks}

In order to state our main theorem it is useful to introduce some
terminology. In what follows, we fix an ordering on the simple roots as
in Table~\ref{Dynkin}; our conventions agree with that in the
standard textbook of Humphreys \cite{Humphreys}. In particular, given
an index $i$ with $1 \leq i \leq r$ where $r$ is the rank of $G$, we
may refer to its corresponding simple reflection $s_i :=
s_{\alpha_i}$, where the index $i$ refers to the ordering of the roots
in Table~\ref{Dynkin}. 

\begin{center}
  \begin{tabular}[h]{|c|c|}
  \hline
  $\Phi$ & Dynkin diagram \\
  \hline
  \hline
  $A_r$ $(r \geq 1)$ &
\begin{picture}(110,15)
	\multiput(5,5)(25,0){5}{\circle{5}}
	\multiput(7.5,5)(25,0){2}{\line(1,0){20}}
	\multiput(57.5,5)(8,0){3}{\line(1,0){4}}
	\put(82.5,5){\line(1,0){20}}
	
	\put(3,-7){\small{$1$}}
	\put(27,-7){\small{$2$}}
	\put(52,-7){\small{$3$}}
	\put(70,-7){\small{$r-1$}}
	\put(102,-7){\small{$r$}}
\end{picture}
  \\
  $B_r$ $(r \geq 2)$ &
\begin{picture}(110,20)
	\multiput(5,5)(25,0){5}{\circle{5}}
	\multiput(32.5,5)(8,0){3}{\line(1,0){4}}
	\put(7.5,5){\line(1,0){20}}
	\put(57.5,5){\line(1,0){20}}
	\multiput(82.5,4)(0,2){2}{\line(1,0){20}}
	
	\put(90,2.5){$>$}
	
	\put(3,-7){\small{$1$}}
	\put(27,-7){\small{$2$}}
	\put(45,-7){\small{$r-2$}}
	\put(70,-7){\small{$r-1$}}
	\put(102,-7){\small{$r$}}
\end{picture}
  \\
  $C_r$ $(r \geq 3)$ &  
\begin{picture}(110,20)
	\multiput(5,5)(25,0){5}{\circle{5}}
	\multiput(32.5,5)(8,0){3}{\line(1,0){4}}
	\put(7.5,5){\line(1,0){20}}
	\put(57.5,5){\line(1,0){20}}
	\multiput(82.5,4)(0,2){2}{\line(1,0){20}}
	
	\put(90,2.5){$<$}
	
	\put(3,-7){\small{$1$}}
	\put(27,-7){\small{$2$}}
	\put(45,-7){\small{$r-2$}}
	\put(70,-7){\small{$r-1$}}
	\put(102,-7){\small{$r$}}
\end{picture}
  \\
  $D_r$ $(r \geq 4)$ & 
\begin{picture}(110,35)
	\multiput(5,5)(25,0){4}{\circle{5}}
	\put(105,15){\circle{5}}
	\put(105,-5){\circle{5}}
	\multiput(32.5,5)(8,0){3}{\line(1,0){4}}
	\put(7.5,5){\line(1,0){20}}
	\put(57.5,5){\line(1,0){20}}
	\put(82.3,7.3){\line(5,2){20}}
	\put(82.3,2.7){\line(5,-2){20}}	

	\put(3,-7){\small{$1$}}
	\put(27,-7){\small{$2$}}
	\put(45,-7){\small{$r-3$}}
	\put(70,-7){\small{$r-2$}}
	\put(97,5){\small{$r-1$}}
	\put(102,-15){\small{$r$}}
\end{picture}
  \\
  $E_6$ & 
\begin{picture}(110,40)
	\multiput(5,5)(25,0){5}{\circle{5}}
	\multiput(7.5,5)(25,0){4}{\line(1,0){20}}
	\put(55,25){\circle{5}}
	\put(55,7.5){\line(0,1){15}}
	
	\put(3,-7){\small{$1$}}
	\put(27,-7){\small{$3$}}
	\put(52,-7){\small{$4$}}
	\put(77,-7){\small{$5$}}
	\put(102,-7){\small{$6$}}
	\put(60,20){\small{$2$}}
\end{picture}
  \\
  $E_7$ & 
\begin{picture}(135,40)
	\multiput(5,5)(25,0){6}{\circle{5}}
	\multiput(7.5,5)(25,0){5}{\line(1,0){20}}
	\put(55,25){\circle{5}}
	\put(55,7.5){\line(0,1){15}}
	
	\put(3,-7){\small{$1$}}
	\put(27,-7){\small{$3$}}
	\put(52,-7){\small{$4$}}
	\put(77,-7){\small{$5$}}
	\put(102,-7){\small{$6$}}
	\put(127,-7){\small{$7$}}
	\put(60,20){\small{$2$}}
\end{picture}
 \\
  $E_8$ & 
\begin{picture}(160,40)
	\multiput(5,5)(25,0){7}{\circle{5}}
	\multiput(7.5,5)(25,0){6}{\line(1,0){20}}
	\put(55,25){\circle{5}}
	\put(55,7.5){\line(0,1){15}}
	
	\put(3,-7){\small{$1$}}
	\put(27,-7){\small{$3$}}
	\put(52,-7){\small{$4$}}
	\put(77,-7){\small{$5$}}
	\put(102,-7){\small{$6$}}
	\put(127,-7){\small{$7$}}
	\put(152,-7){\small{$8$}}
	\put(60,20){\small{$2$}}
\end{picture}
  \\
  $F_4$ & 
\begin{picture}(85,20)
	\multiput(5,5)(25,0){4}{\circle{5}}
	\put(7.5,5){\line(1,0){20}}
	\put(57.5,5){\line(1,0){20}}
	\multiput(32.5,4)(0,2){2}{\line(1,0){20}}
	
	\put(40,2.5){$>$}
	
	\put(3,-7){\small{$1$}}
	\put(27,-7){\small{$2$}}
	\put(52,-7){\small{$3$}}
	\put(77,-7){\small{$4$}}
\end{picture}
  \\
  \raisebox{3mm}{$G_2$} & 
\begin{picture}(35,30)
	\multiput(5,15)(25,0){2}{\circle{5}}
	\multiput(7.5,13.5)(0,1.5){3}{\line(1,0){20}}
	
	\put(15,12.5){$<$}
	
	\put(3,3){\small{$1$}}
	\put(27,3){\small{$2$}}
\end{picture}
  \\
  \hline
\end{tabular}
\captionof{table}{Dynkin diagrams for all Lie types.}\label{Dynkin}
\end{center}

\begin{definition}
  Let $\mathcal{I} = (i_1, i_2, \ldots, i_n) \in [r]^n$ be a (not necessarily
  reduced) word decomposition of an element $w = s_{i_1} s_{i_2} \cdots s_{i_n}$ of the Weyl group
  $W$. We say that
  $\mathcal{I}$ is a \textbf{diagram walk} (or \textbf{walk})
  if successive simple roots are adjacent in the corresponding Dynkin
  diagram, or more precisely, for each $j \in [n-1] = \{1 \leq j \leq
  n-1\}$, the two successive roots $\al_{i_j}$ and $\al_{i_{j+1}}$ are
  distinct, and 
 there is an edge in the corresponding Dynkin diagram connecting
 $\alpha_{i_j}$ and $\alpha_{i_{j+1}}$. We call $i_1$ (or
  $\alpha_{i_1}$) the \textbf{initial root (of the diagram walk
    $\mathcal{I}$)} and denote it by $IR(\mathcal{I})$. We call $i_n$ (or $\alpha_{i_n}$)
  the \textbf{final root (of the diagram walk $\mathcal{I}$)} and
  denote it $FR(\mathcal{I})$. 
\end{definition}

\begin{example}\label{example:diagram walks}
\begin{enumerate} 
  \item In type A, the words $s_2 s_3 s_4 s_5 s_4 s_3$ and $s_1 s_2
    s_1 s_2 s_3$ are both diagram walks. Note that the second word is
    not reduced. 
 \item In type B, $s_{r-2} s_{r-1} s_r$ is a diagram walk. 
  \item In type $E_8$, $s_1s_3s_4s_2s_4s_5$ is a diagram walk.  
\end{enumerate}
\end{example}

In what follows, we also find it useful to consider words which are
`almost' diagram walks, except that the word begins with a repetition
(thus disqualifying it from being a walk), i.e. the initial root
appears twice. 

\begin{definition}
  Let $\mathcal{I} = (i_0, i_1, i_2, \ldots, i_n)$ be a (not necessarily
  reduced) word decomposition of an element $w = s_{i_0} s_{i_1} \cdots s_{i_n}$ of the Weyl group
  $W$. We say that
  $\mathcal{I}$ is a \textbf{hesitant (diagram) walk} if 
  \begin{itemize}
  \item the length of the word is at least $2$, i.e. $n \geq 1$, 
  \item the first two roots are the same, i.e., $i_0 = i_1$, and 
  \item the subword $(i_1, \ldots, i_n)$ is a diagram walk. 
  \end{itemize}
  In other words, except for the `hesitation' at the first step, the
  remainder of the word is a diagram walk. We refer to the subword $(i_1, \ldots,
  i_n)$ as the \textbf{walking component} of the hesitant walk. 
\end{definition}

A few remarks are in order. First, we emphasize that a hesitant walk,
despite the terminology, is not actually a diagram walk; it becomes a diagram
walk only after deleting the first entry in the word. Furthermore, it is clear that a hesitant
(diagram) walk is never a reduced word decomposition (because of the
two repeated roots at the beginning). 
On the other hand, it is possible for a reduced word decomposition to
\emph{contain} a hesitant walk as a subword: for instance, for $G =
SL(4,\C)$, the reduced word decomposition $s_1s_2s_3s_1s_2s_1$ for the longest
element in the Weyl group $S_4$ contains $s_1s_1 s_2$ as a subword, which is a hesitant
walk. 

\begin{definition}
  Let $\mathcal{I} = (i_1, i_2, \ldots, i_n)$ be a
  word decomposition of an element $w = s_{i_1} s_{i_2} \cdots
  s_{i_n}$ of the Weyl group
  $W$. We say that
  $\mathcal{I}$ is \textbf{hesitant-walk-avoiding} if there is no 
  subword $\mathcal{J} = (i_{j_0}, i_{j_1}, \ldots,
  i_{j_s})$ of $\mathcal{I}$ which is a hesitant walk. 
  \end{definition}
  
  \begin{example}
    Let $G = SL(4,\C)$ with Weyl group $S_4$. The reduced
    word decomposition $s_1s_2s_3$ is hesitant-walk-avoiding.
 \end{example}
    
In what follows we will also be interested in dominant weights
$\lambda$ in the character lattice $X(T)$ associated to $G$. As in
Section~\ref{sec:background}, we may express $\lambda$ as a linear combination of
the fundamental weights $\varpi_1, \ldots, \varpi_r$ corresponding to
the simple roots $\alpha_1, \ldots, \alpha_r$. Thus we write 
\[
\lambda = \lambda_1 \varpi_1 + \cdots + \lambda_r \varpi_r
\]
and since we assume $\lambda$ is dominant, $\lambda_i \geq 0$ for all $i=1,
\ldots, r$. 

\begin{definition}
  Let $\lambda$ be as above. We say that a simple root $\alpha_i$
  \textbf{appears in $\lambda$} if the corresponding coefficient is
  strictly positive, i.e. 
  \begin{equation}
    \label{eq:appears in lambda}
    \lambda_i = \langle \lambda, \alpha_i^{\vee} \rangle > 0. 
  \end{equation}
\end{definition}

We now introduce some terminology which relates diagram walks and
hesitant walks to the dominant weight $\lambda$. 

\begin{definition}
  Let $\lambda$ and $\mathcal{I}$ be as above. 
 We will say that $\mathcal{I}$ is a \textbf{$\lambda$-walk} if
  \begin{itemize}
  \item $\mathcal{I}$ is a diagram walk, and 
  \item the final root $FR(\mathcal{I})$ of the walk $\mathcal{I}$
    appears in $\lambda$. 
  \end{itemize}
Similarly, we say that $\mathcal{I}$ is a \textbf{hesitant
  $\lambda$-walk} if it is a hesitant walk, and the final root of its
walking component appears in $\lambda$. Finally, a word $\mathcal{I}$
is \textbf{hesitant-$\lambda$-walk-avoiding} if there is no subword 
$\mathcal{J}$ of $\mathcal{I}$ which is a hesitant $\lambda$-walk. 
\end{definition}

 \begin{example}
    Let $G = SL(4,\C)$ with Weyl group $S_4$. Consider the reduced
    word decomposition $\mathcal{I} = (1,2,3,1,2,1)$ of the 
    longest element $w_0 = s_1 s_2 s_3 s_1 s_2 s_1$ of
    $S_4$ and $\lambda = 3\varpi_3$. Then $\mathcal{I}$ is
hesitant-$\lambda$-walk-avoiding. 
  \end{example}

Given the terminology introduced above we may now state our 
main theorem. 

\begin{theorem}\label{theorem:main}
Let $\mathcal{I} = (i_1, i_2, \ldots, i_n)$ be a word decomposition of
an element 
$w=s_{i_1}\cdots s_{i_n}$ of $W$ and let 
$\lambda = \lambda_1 \varpi_1 + \lambda_2 \varpi_2 + \cdots + \lambda_r \varpi_r$ be a
dominant weight. Let $\mfc = \{c_{jk}\}$ and $\mfl = (\ell_1, \ldots,
\ell_n)$ be determined from $\lambda$ and $\mathcal{I}$ as in~\eqref{eq:def cjk rep} and~\eqref{eq:def
  ell rep}. 
Then the corresponding Grossberg-Karshon twisted cube $(C(\mfc, \mfl),
\rho)$ is untwisted if
and only if $\mathcal{I}$ is hesitant-$\lambda$-walk-avoiding. 

\end{theorem}

The proof of the above theorem occupies Sections~\ref{sec:necessity}
and~\ref{sec:sufficiency}.

\section{Proof of the main theorem: necessity}\label{sec:necessity}

We begin the proof of Theorem~\ref{theorem:main} 
by first proving the ``only if'' part of the
statement, i.e., that hesitant-$\lambda$-walk-avoidance implies the
untwistedness of the Grossberg-Karshon twisted cube. 

We need some preliminary lemmas. Recall that the $m_\sigma
 = (m_{\sigma,1}, \ldots, m_{\sigma,n})$ are
integer vectors defined by~\eqref{eq:def m_sigma}
associated to the defining constants $\mfc$ and $\mfl$ of the twisted cube.

\begin{lemma}\label{lemma:existence of lambda walk}
Let $\{c_{ij}\}_{1 \leq i < j \leq n}$ and 
$\ell_1,\dots,\ell_n$ be fixed integers. 
Assume that $\ell_i \geq 0$ for all $i$.
If there exists an element $\sigma$ of $\{+,-\}^n$ and $k \in [n]$ 
such that $m_{\sigma, k} > 0$ and $m_{\sigma, i} \geq 0$ for 
$i > k$, then there exists an increasing sequence $\mathcal{J}$ of indices $1 \leq
j_1 < j_2 < \cdots < j_s \leq n$, with $s \geq 1$, such that 
\begin{enumerate} 
\item ${j_1} = k$,
\item $\ell_{j_s} > 0$, and 
\item  $c_{{j_t} j_{t+1}}  < 0$ for $t = 1,\dots,s-1$. 
\end{enumerate} 
\end{lemma}

\begin{proof}
Let $\sigma$ and $k$ be as above. 
We may explicitly construct the subsequence $\mathcal{J}$ as
follows. First suppose $\ell_k > 0$. In this case, the subsequence
$\mathcal{J} = (j_1 = k)$ satisfies the required three conditions (the
third being vacuous), so we are done. If on the other hand $\ell_k =
0$, we set $j_1 = k$ and then define $j_2$ as follows. By assumption
$m_{\sigma,k} > 0$ so we know $\sigma_k = -$, and by definition of the
$m_\sigma$ we know 
\begin{equation}
\begin{split}
m_{\sigma,k} & = \ell_k - \sum_{i > k} c_{ki} m_{\sigma,i} \\
 & = - \sum_{i > k} c_{ki} m_{\sigma,i}.
\end{split}
\end{equation}
Since $m_{\sigma,i} \geq 0$ for $i \geq k$ by assumption, in order for
$m_{\sigma,k}$ to be strictly positive there must exist an index $J >
k$ with $c_{kJ} < 0$ and $m_{\sigma, J} >0$. Choose $j_2$ to be the
minimal such index. If $\ell_{j_2} > 0$, then the sequence
$\mathcal{J} = (j_1 = k, j_2)$ satisfies the required three conditions
and we are done. Otherwise, we may repeat the above argument as many
times as necessary (i.e. as long as $\ell_{j_t} = 0$). Since the
indices $j_t$ are bounded above by $n$, this process must stop,
i.e. there must exist some $s \geq 1$ such that the sequence
$\mathcal{J} = (j_1, \ldots, j_s)$ found in this manner satisfies the
requirements. 

\end{proof}

In the case when the constants $\mfc$ and $\mfl$ are obtained from 
the data of a weight $\lambda$ and a word $\mathcal{I}$ we can 
interpret Lemma~\ref{lemma:existence of lambda walk} using the
terminology introduced in Section~\ref{sec:walks}.

\begin{corollary}\label{corollary:existence of lambda walk}
Let $\mathcal{I} = (i_1, i_2, \ldots, i_n)$ be a word decomposition of
an element 
$w = 
s_{i_1}\cdots s_{i_n}$ of $W$ and let 
$\lambda = \lambda_1 \varpi_1 + \lambda_2 \varpi_2 + \cdots + \lambda_r \varpi_r$ be a
dominant weight, i.e. $\lambda_i \geq 0$ for all $i$. 
Let $\mfc, \mfl$ and
$\{m_{\sigma}\}_{\sigma \in \{+,-\}^n}$ be determined from
$\mathcal{I}$ and $\lambda$ as in~\eqref{eq:def cjk rep},~\eqref{eq:def
  ell rep} and~\eqref{eq:def m_sigma}. 
If there exists an element $\sigma$ of $\{+,-\}^n$ and $k \in [n]$ 
such that $m_{\sigma, k} > 0$ and $m_{\sigma, i} \geq 0$ for 
$i > k$, 
then there exists a subword 
$\mathcal{J}= (i_{j_1},i_{j_2},\dots,i_{j_s})$
of $\mathcal{I}$, of length at least $1$ (i.e. $s \geq 1$), such that 
$j_1 = k$ and $\mathcal{J}$ is a $\lambda$-walk (i.e. it is a diagram
walk and the final root $FR(\mathcal{J})$ appears in $\lambda$). 
\end{corollary}

\begin{proof}
First observe that by the definition~\eqref{eq:def ell rep} of the
$\ell_i$ and by assumption on $\lambda$, we have $\ell_i \geq 0$ for
all $i$, and $\ell_i > 0$ exactly when $\beta_i$, the $i$-th simple
root in $\mathcal{I}$, appears in $\lambda$. Let $\sigma$ and $k$ be
as above. Then 
by Lemma~\ref{lemma:existence of lambda walk} there exists a subword 
$\mathcal{J}= (i_{j_1} = i_k, i_{j_2},\dots,i_{j_s})$ of length at
least $1$ 
such that the $j_1 = k$ and 
$FR(\mathcal{J})$ appears in $\lambda$.
It remains to check that $\mathcal{J}$ is a diagram walk. Recall that by
definition 
$c_{j\ell} = \lee \beta_\ell, \beta_j^{\vee} \ree$.
Hence $c_{j\ell} < 0$ if and only if 
there is an edge in the corresponding Dynkin diagram connecting the
roots $\alpha_{i_j}$ and $\alpha_{i_\ell}$, so by the conditions on
$\mathcal{J}$ in Lemma~\ref{lemma:existence of lambda walk} we see
that $\mathcal{J}$ is a diagram walk, as desired. 
\end{proof}

The next result is the main technical fact we need.

\begin{lemma}\label{lemma:existence hesitant lambda walk}
Let $\{c_{ij}\}_{1 \leq i < j \leq n}$ and 
$\ell_1,\dots,\ell_n$ be fixed integers and let $(C(\mfc, \mfl),
\rho)$ be the corresponding Grossberg-Karshon twisted cube. 
Assume that $\ell_i \geq 0$ for all $i$.
If $(C(\mfc, \mfl),\rho)$ is twisted, then there exists an increasing
subsequence
$\mathcal{J} = (j_0 < j_1 < \cdots < j_s)$ of indices of length at
least $2$ (i.e. $s \geq 1$) such that 
\begin{enumerate} 
\item $\ell_{j_s} > 0$, 
\item $c_{j_0 j_1}>0$, and 
\item $c_{{j_t} j_{t+1}}  < 0$ for all $t = 1,\dots,s-1$. 
\end{enumerate} 
\end{lemma}

\begin{proof}
By Theorem \ref{thm-HY},
there exist an element $\sigma$ of $\{+,-\}^n$ and an index $k$ such that 
$m_{\sigma, k} < 0$. For such a choice of $\sigma$ we may assume
without loss of generality that $k$ is chosen to be the maximal such
index, i.e. that $m_{\sigma,k} < 0$ and $m_{\sigma, s} \geq 0$ for
$s>k$. 
Recall that by definition 
\[
m_{\sigma,k} = \ell_k - \sum_{s>k} c_{ks} m_{\sigma,s}.
\]
By assumption $m_{\sigma,k} < 0$ so we have $\sum_{s>k} c_{ks}
m_{\sigma,s} > \ell_k \geq 0$. Since also $m_{\sigma,s} \geq 0$ for
$s>k$, this implies that there exists some $p > k$ with $c_{kp} > 0$
and $m_{\sigma,p} > 0$. Applying Lemma~\ref{lemma:existence of lambda
  walk} we obtain an increasing sequence $(j_1 = p, j_2, \ldots, j_s)$
of indices with $s \geq 1$ such that $\ell_{j_s} > 0$ and $c_{j_t
  j_{t+1}} < 0$ for all $t=1, \ldots, s-1$. Then by choosing $j_0 = k
< j_1 = p$ and since $c_{j_0 j_1} = c_{kp} > 0$ by construction of
$p$, we obtain a sequence $\mathcal{J} = (j_0 =k, j_1=p, \ldots, j_s)$
satisfying the required conditions.

\end{proof}

The proof of
the ``only if'' part of Theorem~\ref{theorem:main} is a
straightforward consequence of the above lemma.

\begin{proof}[Proof of the ``only if" part of Theorem \ref{theorem:main}]\label{onlyif}
Suppose the Grossberg-Karshon twisted cube $(C(\mfc, \mfl),\rho)$ is
twisted. By the dominance assumption on $\lambda$ and by definition of the
$\ell_i$, we know $\ell_i \geq 0$ for all $i$. Thus we
may apply Lemma~\ref{lemma:existence hesitant lambda walk}. 
Note also $\ell_{j_s} > 0$ precisely when the root $\beta_{j_s}$ appears in $\lambda$. 
Moreover, by definition, we know that $c_{j_0 j_1} := \lee \beta_{j_1}, \beta_{j_0}^{\vee}
\ree >0$ if and only if $\beta_{j_0} = \beta_{j_1}$ (equivalently $i_{j_0}
= i_{j_1}$) and 
$c_{j_t j_{t+1}} < 0$ if and only if there is an edge in the corresponding
Dynkin diagram connecting the roots $\beta_{j_t}$ and
$\beta_{j_{t+1}}$.
Thus the subword 
$(i_{j_0}, i_{j_1}, \dots,i_{j_s})$ of $\mathcal{I}$ corresponding to the
subsequence $(j_0, j_1, \ldots, j_s)$ of indices obtained from
Lemma~\ref{lemma:existence hesitant lambda walk} is a hesitant
$\lambda$-walk, as desired. 
\end{proof}

\section{Proof of the main theorem: sufficiency}\label{sec:sufficiency}

We now proceed to prove the ``if'' part of the main theorem, i.e., that
untwistedness implies hesitant-$\lambda$-walk-avoidance. 
Part of the proof will be a case-by-case analysis of the possible Lie types of $G$. 

For convenience we recall the Cartan matrices for all
Lie types (see for example \cite[p.58-59]{Humphreys}). 
\begin{center}
\small{
\begin{tabular}[h]{ll}
%\hline
$
A_r:~ \begin{bmatrix}
2 & -1 & 0 & & . & . & . & & & 0 \\
-1 & 2 & -1 & 0 & . & . & . & & & 0 \\
0 & -1 & 2 & -1 & 0 & . & . & . & & 0 \\
. & . & . & . & . & . & . & . & . & . \\
0 & 0 & 0 & & . & . & . & -1 & 2 & -1 \\ 
0 & 0 & 0 & & . & . & . & 0 & -1 & 2 
\end{bmatrix}$&
$
E_6:~\begin{bmatrix}
2 & 0 & -1 & 0 & 0 & 0 \\
0 & 2 & 0 & -1 & 0 & 0 \\
-1 & 0 & 2 & -1 & 0 & 0 \\
0 & -1 & -1 & 2 & -1 & 0 \\
0 & 0 & 0 & -1 & 2 & -1 \\
0 & 0 & 0 & 0 & -1 & 2
\end{bmatrix} $ \\
$B_r:~ \begin{bmatrix}
2 & -1 & 0 & & . & . & . & & & 0 \\
-1 & 2 & -1 & 0 & . & . & . & & & 0 \\
. & . & . & . & . & . & . & . & . & . \\
0 & 0 & 0 & & . & . & . & -1 & 2 & -2 \\
0 & 0 & 0 & & . & . & . & 0 & -1 & 2 
\end{bmatrix} $ &
$E_7:~\begin{bmatrix}
2 & 0 & -1 & 0 & 0 & 0 & 0\\
0 & 2 & 0 & -1 & 0 & 0 & 0\\
-1 & 0 & 2 & -1 & 0 & 0 & 0\\
0 & -1 & -1 & 2 & -1 & 0 & 0\\
0 & 0 & 0 & -1 & 2 & -1 & 0\\
0 & 0 & 0 & 0 & -1 & 2 & -1 \\
0 & 0 & 0 & 0 & 0 & -1 & 2 \\
\end{bmatrix} $ \\
$C_r:~ \begin{bmatrix}
2 & -1 & 0 & & . & . & . & & & 0 \\
-1 & 2 & -1 & 0 & . & . & . & & & 0 \\
. & . & . & . & . & . & . & . & . & . \\
0 & 0 & 0 & & . & . & . & -1 & 2 & -1 \\
0 & 0 & 0 & & . & . & . & 0 & -2 & 2 
\end{bmatrix} $ &
$E_8:~\begin{bmatrix}
2 & 0 & -1 & 0 & 0 & 0 & 0 & 0\\
0 & 2 & 0 & -1 & 0 & 0 & 0 & 0\\
-1 & 0 & 2 & -1 & 0 & 0 & 0 & 0\\
0 & -1 & -1 & 2 & -1 & 0 & 0 & 0\\
0 & 0 & 0 & -1 & 2 & -1 & 0 & 0\\
0 & 0 & 0 & 0 & -1 & 2 & -1 & 0\\
0 & 0 & 0 & 0 & 0 & -1 & 2 & -1\\
0 & 0 & 0 & 0 & 0 & 0 & -1 & 2\\
\end{bmatrix} $ \\
$D_r:~ \begin{bmatrix}
2 & -1 & 0 & & . & . & . & & & 0 \\
-1 & 2 & -1 &  & . & . & . & & & 0 \\
. & . & . & . & . & . & . & . & . & . \\
0 & 0 & & . & . & -1 & 2 & -1 & 0 & 0 \\
0 & 0 & & . & . & & -1 & 2 & -1 & -1 \\
0 & 0 & & . & . & & 0 & -1 & 2 & 0 \\
0 & 0 & & . & . & & 0 & -1 & 0 & 2  
\end{bmatrix} $ &
$F_4:~\begin{bmatrix}
2 & -1 & 0 & 0 \\
-1 & 2 & -2 & 0\\
0 & -1 & 2 & -1 \\
0 & 0 & -1 & 2
\end{bmatrix}
\qquad 
G_2:~\begin{bmatrix}
2 & -1 \\
-3 & 2
\end{bmatrix} $\\
%\hline
\end{tabular}
}
\captionof{table}{Cartan matrices for all Lie types.}\label{Cartan matrices}
\end{center}

In the discussion below it will be useful to restrict attention 
to hesitant $\lambda$-walks which are minimal in an appropriate sense.
We make this precise in the definition below. 

\begin{definition}\label{def:minimal}
Let $\lambda$ be a dominant weight and let 
$\mathcal{I} = (i_0, \ldots, i_n)$ be a hesitant $\lambda$-walk. 
We say that 
  $\mathcal{I}$ is \textbf{minimal} 
   if 
\begin{enumerate} 
\item $\{i_1, \ldots, i_n\}$ are all
  distinct, i.e. the walking component of $\mathcal{I}$ visits any given vertex of the
  Dynkin diagram at most once, and 
\item if $n \geq 2$, then $\beta_{0}, \ldots, \beta_{{n-1}}$ do not appear in
  $\lambda$. 
\end{enumerate}  
\end{definition}

\begin{example}
Let $G = SL(6,\C)$. 
\begin{itemize} 
\item Let $\lambda = \varpi_2$. The hesitant
$\lambda$-walk $\mathcal{J} = (5, 5, 4, 3, 4, 3, 2)$ is not minimal
since the walking component revisits some vertices multiple times,
but the subword $\mathcal{J}' = (5, 5, 4, 3, 2)$ is minimal. 
\item Let $\lambda = \varpi_2 + \varpi_5$. In this case the above
  hesitant $\lambda$-walk $(5,5,4,3,2)$ is not minimal since
  $\beta_{0} = \beta_{1} = \al_5$ already appears in
  $\lambda$. The subword $(5,5)$ is minimal. 
\end{itemize} 
\end{example} 

It is clear from the definition that for any dominant
$\lambda \neq 0$ and a hesitant $\lambda$-walk $\mathcal{J}$, there
exists a subword $\mathcal{J}'$ of $\mathcal{J}$ which is minimal in
the sense of Definition~\ref{def:minimal}.

\begin{lemma}\label{lemma:minimality implications}
Let $\lambda \neq 0$ be a dominant weight and $\mathcal{J} = (i_{j_0},
i_{j_1}, \ldots, i_{j_s})$ a 
hesitant $\lambda$-walk. Let $\mfc, \mfl$ be the constants associated
to $\mathcal{J}$ and $\lambda$ as defined in~\eqref{eq:def cjk rep}
and~\eqref{eq:def ell rep}. 
If $\mathcal{J}$ is minimal, then 
\begin{enumerate} 
\item $c_{j_p j_q} = 0$ if $|p - q | \geq 2$ and $1 \leq p,q \leq s$,
  and
\item $\ell_{j_p}  = 0$ for $0 \leq p \leq s-1$ if $s \geq 2$. 
\end{enumerate} 
\end{lemma} 

\begin{proof} 
By the minimality assumption, and since Dynkin diagrams have no loops, 
we know that if $\lvert p - q \rvert \geq 2$ and $1
\leq p,q \leq s$ (so $j_p$ and $j_q$ are in the walking component of
$\mathcal{J}$) then the roots $\beta_{j_p}$ are neither adjacent nor
equal. This implies the corresponding entry in the Cartan matrix is
$0$, as desired. 
The second statement is immediate from the minimality assumption 
since $\ell_{j_p} > 0$ exactly when
$\beta_{j_p}$ appears in $\lambda$. 
\end{proof}

We begin with a lemma. 

\begin{lemma}\label{lemma:length 2 hesitant walk implies twisted}
Let $\{c_{ij}\}_{1 \leq i < j \leq n}$ and 
$\ell_1,\dots,\ell_n$ be fixed integers and let $(C(\mfc, \mfl),
\rho)$ be the corresponding Grossberg-Karshon twisted cube. 
Assume that $\ell_i \geq 0$ for all $i$. 
If there exists two distinct indices $i$ and $j$, $1 \leq i < j \leq
n$, with $c_{ij} > 1$ and $\ell_i = \ell_j > 0$, then $(C(\mfc, \mfl),
\rho)$ is twisted. 
\end{lemma}

\begin{proof}
By Theorem~\ref{thm-HY}, it suffices to show that there exists
an element $\sigma$ of $\{+,-\}^n$ and some $k$ with $1 \leq k \leq n$
such that $m_{\sigma, k} < 0$. 
Let $\sigma = (\sigma_1,\dots,\sigma_n) \in \{+,-\}^n$ be the 
element defined by
\[
\sigma_k = 
\begin{cases}
- & \textrm{ if } k = i \textrm{ or } j \\
+ & \textrm{ otherwise }
\end{cases}
\]
and consider the associated $m_\sigma = (m_{\sigma,1}, \ldots,
m_{\sigma,n})$. Then by definition of $\sigma$ and $m_{\sigma}$ we
have 
\begin{align*}
m_{\sigma,j} &= \ell_j - \sum_{s > j} c_{js}m_{\sigma,s} \\
m_{\sigma,i} &= \ell_i - \left( c_{ij}m_{\sigma,j} - \sum_{\substack{s
      > i \\  s \neq j}} c_{is}m_{\sigma,s} \right).
\end{align*}
Since $\sigma_k = +$ for $k \neq i,j$, we have that $m_{\sigma,k} = 0$ for 
$k \neq i,j$. Hence the above equations can be simplified to 
\begin{align*}
m_{\sigma,j} &= \ell_j \\
m_{\sigma,i} &= \ell_i - c_{ij}m_{\sigma,j} = \ell_i - c_{ij} \ell_j 
\end{align*}
By assumption $\ell_i = \ell_j$ so 
\[
m_{\sigma, i} = \ell_i(1 - c_{ij}). 
\]
Since $c_{ij} > 1$ and $\ell_i > 0$, we obtain $m_{\sigma,i} < 0$, as
desired. 
\end{proof}

As in the previous section, the above lemma can be interpreted in
terms of hesitant $\lambda$-walks. 

\begin{corollary}\label{corollary:length 2 hesitant walk implies twisted}
Let $\mathcal{I} = (i_1, i_2, \ldots, i_n)$ be a word decomposition of
an element 
$w = 
s_{i_1}\cdots s_{i_n}$ of $W$ and let 
$\lambda = \lambda_1 \varpi_1 + \lambda_2 \varpi_2 + \cdots + \lambda_r \varpi_r$ be a
dominant weight, i.e. $\lambda_i \geq 0$ for all $i$. 
Let $\mfc = \{c_{jk}\}, \mfl = (\ell_1, \ldots, \ell_n)$ and
$\{m_{\sigma}\}_{\sigma \in \{+,-\}^n}$ be determined from
$\mathcal{I}$ and $\lambda$ as in~\eqref{eq:def cjk rep},~\eqref{eq:def
  ell rep} and~\eqref{eq:def m_sigma} and let $(C(\mfc, \mfl), \rho)$
denote the corresponding Grossberg-Karshon twisted cube. If
$\mathcal{I}$ contains a subword $\mathcal{J} = (j_0, j_1)$ of length
$2$ which is a
hesitant $\lambda$-walk, then $(C(\mfc, \mfl),\rho)$ is twisted. 
\end{corollary} 

\begin{proof} 
By definition of a hesitant $\lambda$-walk, if $\mathcal{J} = (j_0,
j_1)$ is a hesitant $\lambda$-walk then $i_{j_0}=i_{j_1}$
(equivalently $\beta_{j_0} = \beta_{j_1}$) and
$\beta_{j_0}=\beta_{j_1}$ appears in $\lambda$. This implies $c_{j_0
  j_1}=2>1$ and $\ell_{j_0} = \ell_{j_1} > 0$. The result now follows
from Lemma~\ref{lemma:length 2 hesitant walk implies twisted}. 
\end{proof}

\begin{proof}[Proof of ``if" part of Theorem \ref{theorem:main}]
Suppose $\mathcal{J} = \{i_{j_0}, i_{j_1}, \cdots,i_{j_s}\}$ is a subword of
$\mathcal{I}$ which is a hesitant $\lambda$-walk. We may without loss
of generality assume that $\mathcal{J}$ is minimal in the sense of
Definition~\ref{def:minimal}. We then wish to show that $(C(\mfc,
\mfl),\rho)$ is twisted.  If the length of $\mathcal{J}$ is 2,
i.e. $s=1$, then this follows from 
Corollary~\ref{corollary:length 2 hesitant walk implies
twisted}. 
Thus we may now assume that the length is at least $3$, i.e. $s \geq
2$. 
To prove that $(C(\mfc, \mfl),\rho)$ is twisted, 
by Theorem \ref{thm-HY} it is 
enough to find an element $\sigma$ of $\{+,-\}^n$ and a $k \in [n]$
such that $m_{\sigma,k} < 0$. To achieve this, consider the element
$\sigma = (\sigma_1, \ldots, \sigma_n) \in \{+,-\}^n$ defined by 
\[
\sigma_p = 
\begin{cases}
- & \textrm{ if } p \in \{j_0, j_1,\dots,j_s\} \\
+ & \textrm{ otherwise.}
\end{cases}
\]
By definition of $m_\sigma$, we then have 
\begin{equation}\label{eq:find m negative} 
\begin{split}
m_{\sigma,j_{s}} &= \ell_{j_s} - \sum_{p > j_{s}} c_{j_{s} p} m_{\sigma,p} \\
m_{\sigma,j_t} &= \ell_{j_t} - \left( c_{j_t j_{t+1}}m_{\sigma, j_{t+1}} + \sum_{%
\substack{p > j_{t} \\ p \neq j_{t+1}}} c_{j_t p} m_{\sigma, p} \right) %
\textrm{ for } 1 \leq t \leq s-1 \\
m_{\sigma,j_0} &= \ell_{j_0} - %
\left( c_{j_0 j_1}m_{\sigma, j_1} + c_{j_0 j_2}m_{\sigma, j_2} + \sum_{%
\substack{ p > j_0 \\ p \neq j_1,j_2}} c_{j_0 p} m_{\sigma, p} \right).
\end{split} 
\end{equation}
Since $\mathcal{J}$ is a hesitant $\lambda$-walk, we know $\ell_{j_s}
> 0$. On the other hand, by the minimality assumption on $\mathcal{J}$
and Lemma~\ref{lemma:minimality implications}, we know $\ell_{j_t} =
0$ for all $t$ with $0 \leq t \leq s-1$. Moreover, again by minimality
and Lemma~\ref{lemma:minimality implications} we know that $c_{j_t
  j_r} = 0$ for $j_r > j_t$ and $j_r \neq j_{t+1}$. Also by
construction of the $\sigma$, for $p \not \in \mathcal{J} = \{j_0,
j_1, \ldots, j_s\}$ we have $\sigma_p = +$ and hence $m_{\sigma, p} =
0$. Finally, since $\mathcal{J}$ is a hesitant $\lambda$-walk, we have
$\beta_{j_0} = \beta_{j_1}$ and hence $c_{j_0 j_1} = \lee \beta_{j_0},
\beta_{j_1}^{\vee} \ree = 2$. From these considerations we can
simplify~\eqref{eq:find m negative} to 
\begin{equation}
  \label{eq-ifpart}
  \begin{split}
    m_{\sigma,j_s} & = \ell_{j_s} > 0 \\
    m_{\sigma,j_t} & = - c_{j_t j_{t+1}} m_{\sigma, j_{t+1}} \textup{ for } 1
    \leq t \leq s-1 \\
    m_{\sigma,j_0} & = - \left( 2 m_{\sigma, j_1} + c_{j_0 j_2} m_{\sigma,
        j_2} \right).
  \end{split}
\end{equation}

We now claim that $m_{\sigma,j_0}<0$; as already noted, this suffices to
prove the theorem. 
In order to prove this claim we need to know that values of the
constants $c_{j_t j_{t+1}}$ and $c_{j_0 j_2}$ appearing
in~\eqref{eq-ifpart}. By the assumption that $\mathcal{J}$ is a
hesitant $\lambda$-walk, these constants are equal to the
corresponding entry of the Cartan matrices for simple roots which are
adjacent in the Dynkin diagram. For the case-by-case analysis below we
refer to the list of Dynkin diagrams and
Cartan matrices in Tables~\ref{Dynkin} and~\ref{Cartan matrices}. Suppose first that 
the hesitant
$\lambda$-walk only crosses edges of the form 
\begin{picture}(35,10)                   
                       \put(5,5){\circle{5}}    
                       \put(7.5,5){\line(1,0){20}}
                       \put(30,5){\circle{5}}
\end{picture} 
or, if it crosses a double edge
\begin{picture}(35,10)                   
                       \put(5,5){\circle{5}}    
                       \put(7.5,6){\line(1,0){20}}
                       \put(7.5,4){\line(1,0){20}}
 %                      \qbezier(17,3)(19.5,5)(17,7)
                       \put(30,5){\circle{5}}
                       
%                        \put(3,-7){$i$} 
%                        \put(27,-7){$j$}                                                                    
\end{picture} 
or triple edge 
\begin{picture}(35,10)                   
                       \put(5,5){\circle{5}}  
                       \put(7,6.5){\line(1,0){20}}
                       \put(7.5,5){\line(1,0){20}}
                       \put(7,3.5){\line(1,0){20}}
  %                     \qbezier(18.5,3)(16,5)(18.5,7) 
                       \put(30,5){\circle{5}} 
%                        \put(3,-7){$i$} 
%                        \put(27,-7){$j$}
\end{picture}
then it does so only by going in the direction \emph{agreeing with} 
the arrow drawn on the edge in the 
Dynkin diagram (e.g. in type B, if $i_{j_t} = r-1$ and $i_{j_{t+1}} =
r$ and in type $G$, if $i_{j_t} = 2$ and $i_{j_{t+1}} = 1$). In this
situation, the 
corresponding 
constants $c_{j_t j_{t+1}}$ and $c_{j_0 j_2}$ are all equal
to $-1$. 
So we 
consider this case first. In this setting we have 
\begin{equation} 
\begin{split} 
m_{\sigma,j_s} & = \ell_{j_s} > 0 \\
m_{\sigma,j_t} & = m_{\sigma,j_{t+1}} \textup{ for } 1 \leq t \leq s-1 \\
m_{\sigma,j_0} & = - (2m_{\sigma,j_1} - m_{\sigma,j_2}) 
\end{split} 
\end{equation} 
so $m_{\sigma,j_1} = m_{\sigma,j_2} = \cdots = m_{\sigma,j_s} = \ell_{j_s}$ and $m_{\sigma,j_0} =
- \ell_{j_s} < 0$, as desired.

Next we consider the possibility that the hesitant $\lambda$-walk
crosses a ``double'' edge in a direction \emph{against} the direction
of the arrow on the edge. Since we assume the hesitant $\lambda$-walk
is minimal, it can only cross such an edge once. In particular, in type $B$
(respectively type C) this implies that the hesitant $\lambda$-walk
must be of the form $i_{j_0}=i_{j_1} = r$ and $i_{j_2}=r-1,
i_{j_3}=r-2, \dots, i_{j_s} = r-s+1$ (respectively
$i_{j_0}=i_{j_1}=r-s+1, i_{j_2}=r-s+2, \dots, i_{j_{s-1}} = r-1$ and
$i_{j_s}=r$) for some $s\geq 2$. 
We consider these cases next. 

In type $B$ consider the hesitant $\lambda$-walk of the form 
$i_{j_0}=i_{j_1} = r$ and $i_{j_2}=r-1,
i_{j_3}=r-2, \dots, i_{j_s} = r-s+1$ for some $s \geq 2$. In this case the
equations~\eqref{eq-ifpart} become 
\begin{equation*}
  \begin{split}
    m_{\sigma,j_s} &= \ell_{j_s} > 0 \\
    m_{\sigma,j_{s-1}} &= \cdots = m_{\sigma,j_2} = \ell_{j_s} \\
    m_{\sigma,j_1} &= 2 m_{\sigma,j_2} = 2 \ell_{j_s} \\
    m_{\sigma,j_0} &= - (2 m_{\sigma,j_1} + (-2)m_{\sigma,j_2}) = -2
    \ell_{j_s} < 0
 \end{split}
\end{equation*}
so we obtain $m_{\sigma,j_0}<0$ as desired. In type $C$, consider the
hesitant $\lambda$-walk 
$i_{j_0}=i_{j_1}=r-s+1, i_{j_2}=r-s+2, \dots, i_{j_{s-1}} = r-1$ and
$i_{j_s}=r$ for $s \geq 2$. Note that the case $s=2$ is already covered
in the argument for type B above so we may assume $s \geq 3$. 
It is straightforward to see that here we obtain
from~\eqref{eq-ifpart} that 
$m_{\sigma,j_s}=\ell_{j_s}>0$, $m_{\sigma,j_{s-1}} = \cdots =
m_{\sigma,j_1} = 2 \ell_{j_s}$, and $m_{\sigma,j_0} = - 2 \ell_{j_s} <
0$. Thus $m_{\sigma,j_0} < 0$ as desired. 

The only remaining cases are in the exceptional Lie types F and G, but
many cases
of hesitant $\lambda$-walks in type F are already handled by the
considerations in type B and C above. 
Thus the only remaining cases are: 
$(4,4,3,2,1)$ in type F and $(1,1,2)$ in type G. Both are
straightforward and left to the reader.

\end{proof}

\section{Open questions}\label{sec:questions}

The study of Grossberg-Karshon
twisted cubes is related to representation theory and to the recent
theory of Newton-Okounkov bodies and divided-difference operators on
polytopes. Moreover, in this manuscript we have introduced the notion
of hesitant $\lambda$-walks as well as
hesitant-$\lambda$-walk-avoidance. Below, we briefly mention
some possible avenues for further exploration. 

\begin{enumerate} 

\item The Grossberg-Karshon twisted
  cubes are a special case of the virtual polytopes produced by
  Kiritchenko's divided-difference operators \cite{Kiritchenko}. We may ask whether our
  methods generalize to Kiritchenko's setting to provide combinatorial
  conditions on a dominant weight $\lambda$ and choice of word
  decomposition $\mathcal{I}$ which guarantee that
  the corresponding virtual polytope from Kiritchenko's construction
  is a ``true'' polytope. (See also Kiritchenko's discussion in
  \cite[\S 3.3]{Kiritchenko}.)

% \item The results of this manuscript state that for certain choices
%   of dominant weights 
%   $\lambda$ and word expressions $\mathcal{I}$, the Grossberg-Karshon
%   twisted 
%   polytope encoding the character of the representation $H^0(X_w,
%   i^*{\mathcal L}_\lambda)$ where $X_w$ is the Schubert variety
%   corresponding to $\mathcal{I}$ and $i: X_w \into G/B$ is the
%   inclusion to $G/B$, is a `true' polytope. We may ask: what information about the
%   representation, and about the geometry of $X_w$ and
%   $\mathcal{L}_\lambda$, does this polytope encode? 

\item In the cases when the Grossberg-Karshon twisted polytope is
  untwisted (i.e. it is a ``true'' polytope), it would be of interest
  to study
  the relationship between the Grossberg-Karshon polytope and other
  polytopes appearing in representation theory and Schubert calculus,
  such as Gel'fand-Cetlin polytopes, or (more generally) string
  polytopes, or (even more generally) 
  Newton-Okounkov bodies of Bott-Samelson varieties (see \cite{Kaveh,
    Anderson, Harada-Yang:2014b}).

\item ``Pattern avoidance'' is a
recurring and important theme in the study of Schubert varieties. 
We may ask whether, and how,
hesitant-$\lambda$-walk-avoidance relates to the known results in this
direction \cite{AbeBilley}.

\end{enumerate}

%%%%%%%%%%%%%%%%%%%%%%%%%%%%%%%%%%%%%%%%%%%%%%%%%%%%%%%%%%%%%%%%%%%%%%%%


\begin{thebibliography}{AgaGod05}

\bibitem{AbeBilley} H. Abe and S. Billey, \emph{Consequences of the Lakshmibai-Sandhya Theorem: the ubiquity of permutation patterns in Schubert calculus and related geometry}, ArXiv:1403.4345. 

\bibitem{Anderson} D. Anderson, {\em Okounkov bodies and toric degenerations}, Math. Ann. {\bf 356} (2013), no. 3, 1183-1202.


\bibitem{Grossberg-Karshon}
M. Grossberg and Y. Karshon, \emph{Bott towers, complete% 
integrability, and the extended character of representations}, %
Duke Math. J., 76(1):23-58, 1994.

\bibitem{Harada-Yang:2014a}
M. Harada and J. Yang, \emph{Grossberg-Karshon twisted
  cubes and basepoint-free divisors}, ArXiv:1407.4147. 

\bibitem{Harada-Yang:2014b} 
M. Harada and J. Yang, \emph{Newton-Okounkov bodies of
  Bott-Samelson varieties}, in preparation. 

\bibitem{Humphreys}
J.\ E.\ Humphreys, \emph{Introduction to Lie Algebras and Representation Theory}, Springer-Verlag, New York, 1972.

\bibitem{Kaveh} K. Kaveh, {\em Crystal bases and Newton-Okounkov
    bodies}, Arxiv:1101.1687.

\bibitem{Kiritchenko} V. Kiritchenko, \emph{Divided difference
    operators on polytopes}, Arxiv:1307.7234.

\bibitem{LakshmibaiSandhya} 
V. Lakshmibai and B. Sandhya, \emph{Criterion for smoothness
of Schubert varieties in $SL(n)/B$}. Proc. Indian Acad. Sci. (Math Sci.) 100(1): 45–52, 1990.

\end{thebibliography}
\end{document}